\documentclass[hidelinks]{article}
\usepackage[utf8]{inputenc}

\usepackage{amssymb}
\usepackage{amsmath}
\usepackage{amsfonts}
\usepackage{amsthm}
\usepackage{enumitem}
\usepackage{mathtools}
\usepackage{color}
\usepackage{hyperref}
\hypersetup{breaklinks = true}
\usepackage[hyphenbreaks]{breakurl}
\usepackage{aligned-overset}
\usepackage{fix-cm}
\usepackage{parskip}
\usepackage[british]{babel}
\usepackage[title]{appendix}
\usepackage{tikz}
\usetikzlibrary{positioning}
\usepackage{capt-of}
\usepackage{caption}
\usepackage{subcaption}
\usepackage{tabularx}
\usepackage{xcolor}
\definecolor{lightred}{rgb}{1,0.6,0.6}
\usepackage{xurl}
\usepackage{sectsty}
\allsectionsfont{\boldmath}
\usepackage{csquotes}
\usepackage{chngcntr}
\counterwithin{figure}{section}
\counterwithin{table}{section}
\newtheorem{thm}{Theorem}[section]
\newtheorem{lem}[thm]{Lemma}
\newtheorem{cor}[thm]{Corollary}
\theoremstyle{definition}
\newtheorem{defn}[thm]{Definition}

\theoremstyle{remark}
\newtheorem*{rem}{Remark}
\newcommand{\Z}{\mathbb{Z}}

\newcommand{\de}{\coloneqq}
\makeatletter
\newcommand*{\rom}[1]{\expandafter\@slowromancap\romannumeral #1@}
\makeatother
\DeclareMathOperator{\Aut}{Aut}
\DeclareMathOperator{\Out}{Out}
\DeclareMathOperator{\Inn}{Inn}
\DeclareMathOperator{\st}{st}
\DeclareMathOperator{\lk}{lk}
\DeclareMathOperator{\PSO}{PSO}
\DeclareMathOperator{\PSA}{PSA}
\DeclareMathOperator{\id}{id}
\DeclareMathOperator{\Inv}{Inv}
\DeclareMathOperator{\Per}{Per}
\DeclareMathOperator{\GL}{GL}
\newcommand{\sg}[2]{S_{#2}^{#1}}
\usepackage[style=alphabetic,maxbibnames=10]{biblatex}
\title{Right-angled Artin groups as finite-index subgroups of their outer automorphism groups}
\author{Manuel Wiedmer}
\date{}

\begin{document}

\maketitle

\begin{abstract}
We prove that every right-angled Artin group occurs as a finite-index subgroup of the outer automorphism group of another right-angled Artin group. We furthermore show that the latter group can be chosen in such a way that the quotient is isomorphic to $(\Z/2\Z)^N$ for some $N$. For these, we give explicit constructions using the group of pure symmetric outer automorphisms. Moreover, we need two conditions by Day--Wade and Wade--Brück about when this group is a right-angled Artin group and when it has finite index.
\end{abstract}
\let\thefootnote\relax
\section{Introduction}
Right-angled Artin groups are defined by a presentation using a graph $\Gamma$ and denoted by $A_\Gamma$. They were first introduced by Andreas Baudisch in his paper \cite{Baudisch:introduction} under the name of ``semi-free groups''. In recent years they have been used, among other things, for combinatorial approaches to geometric and topological problems; see for example \cite{introduction:example1} and \cite{introduction:example2}. Right-angled Artin groups can be seen as interpolating between free groups and free abelian groups. Similarly, their outer automorphism groups $\Out(A_\Gamma)$ may be viewed as interpolating between $\GL_n(\Z)$ and $\Out(F_n)$. For examples of how this has been used, one can consider \cite{introduction:example3}, \cite{introduction:example4}, \cite{introduction:example5}, \cite{introduction:example6} and \cite{introduction:example7}.\par
Our aim is to study a topic introduced by Matthew B. Day and Richard D. Wade in \cite{DayWade}, which is the main motivation for this paper. We want to understand finite-index subgroups of $\Out(A_\Gamma)$ that are right-angled Artin groups. Day--Wade ask when the group $\Out(A_\Gamma)$ contains a right-angled Artin group as finite-index subgroup; see \cite[Question 1.1]{DayWade}. To partially answer this question, Day--Wade consider the subgroup of so-called pure symmetric outer automorphisms, denoted by $\PSO(A_\Gamma) \leq \Out(A_\Gamma)$. They give a condition on when it is a right-angled Artin group and describe explicitly which one it is in case this condition is satisfied.\par
We look at the question by Day--Wade from another perspective. Day--Wade fix the graph $\Gamma$ and ask if $\Out(A_\Gamma)$ has another right-angled Artin group $A_\Lambda$ as finite-index subgroup. In contrast, we fix the graph $\Lambda$ and obtain the following main result.
{\renewcommand{\thethm}{A}
\begin{thm}\label{thm:anyraagisfinindsubgrofoutraagintro}
For any graph $\Lambda$, the right-angled Artin group $A_\Lambda$ is a finite-index subgroup of the outer automorphism group of some other right-angled Artin group $A_\Gamma$.
\end{thm}
}
This theorem is later stated in more detail as Theorem \ref{thm:anyraagisfinindsubgrofoutraag}. The question for which graphs $\Lambda$ we can find such a graph $\Gamma$ is already mentioned in \cite[Question 3.4]{WadeBrück}. In that paper, Richard D. Wade and Benjamin Brück study a different topic related to the outer automorphism group of right-angled Artin groups. In particular, they give a condition for when the group of pure symmetric outer automorphisms has finite index in the group of outer automorphisms. We use this condition by Wade--Brück and the one by Day--Wade mentioned above to develop Theorem \ref{thm:anyraagisfinindsubgrofoutraagintro}.\par
To prove this theorem, in Section \ref{chap:preliminaries} we introduce the background needed for our main result: Among other things, we discuss a set of generators for $\Aut(A_\Gamma)$, the subgroup $\PSO(A_\Gamma) \leq \Out(A_\Gamma)$ and the two conditions by Wade--Brück and Day--Wade mentioned above. We then show the main result in Section \ref{chap:construction}. For a given graph $\Lambda$, we construct a graph $\Gamma$ such that $A_\Lambda$ is a finite-index subgroup of $\Out(A_\Gamma)$. More precisely, we show that $A_\Lambda \cong \PSO(A_\Gamma)$ and that $\PSO(A_\Gamma)$ has finite index in $\Out(A_\Gamma)$ using the conditions by Day--Wade and Wade--Brück. Later in that section, we impose an additional condition on the graph $\Gamma$. Namely, it should have no non-trivial graph automorphisms. We show that Theorem \ref{thm:anyraagisfinindsubgrofoutraagintro} still holds with this additional condition by developing the construction further. This leads to an interesting corollary about the structure of the quotient $\Out(A_\Gamma)/\PSO(A_\Gamma)$ and lets us determine the index of $A_\Lambda$ in $\Out(A_\Gamma)$. Finally, we conclude this paper by discussing further questions that could be interesting to study in Section \ref{chap:conclusion}. In Appendix \ref{app:smallgraphs}, we cover some special cases of small graphs for which the constructions of Section \ref{chap:construction} do not work.\par
This paper is a shortened version of the author's Master Thesis at ETH Zurich, which can be found as \cite{Masterthesis}. I would like to express my sincere thanks to the supervisor of this thesis, Prof. Dr. Alessandra Iozzi, for making it possible and to the co-supervisor, Dr. Benjamin Brück, for suggesting this topic and for all his support during this project. He always helped me when I had trouble, answered all questions that came up and was open for many helpful discussions, during the thesis but also while changing it into this paper. Furthermore, I wish to thank the anonymous reviewer for the thorough reading of this text as well as his helpful comments and suggestions for the improvement of it.

\section{Preliminaries}\label{chap:preliminaries}
\subsection{Graphs}
Let $\Gamma$ be a graph. We write $V(\Gamma)$ for the set of vertices of $\Gamma$ and $E(\Gamma)$ for the set of edges of $\Gamma$, which is a set of unordered pairs of different vertices. In particular, all the graphs we consider are undirected and do not contain loops or multiple edges. Moreover, we only consider finite graphs, i.e. $V(\Gamma)$ is always a finite set. We write $v \sim w$ if $v$ is adjacent to $w$. For $S \subseteq V(\Gamma)$, we use the notation $\Gamma - S$ for the induced subgraph with vertex set $V(\Gamma) \setminus S$. When we talk of a component of a graph we always mean a connected component. Furthermore, we use $\lk(v)$ for the set of neighbours of $v$ and $\st(v)$ for the union of $\lk(v)$ with $v$ itself.\par

\subsection{Right-angled Artin groups}\label{sec:generators}
\begin{defn}\label{def:raag}
For a given non-empty graph $\Gamma$, the right-angled Artin group $A_\Gamma$ has the following presentation
\[A_\Gamma \de \left\langle V(\Gamma) \mid [v,w] = 1 \text{ for } \{v,w\} \in E(\Gamma) \right\rangle.\]
\end{defn}
Recall that $[v,w] \de vwv^{-1}w^{-1}$ is a notation for the commutator of $v$ and $w$. So, the generators of $A_\Gamma$ correspond to the vertices of $\Gamma$ and two generators commute if there is an edge between the two corresponding vertices in the graph $\Gamma$.\par
Next, we give a set of generators for the automorphism group of a right-angled Artin group $A_\Gamma$. This is based on \cite[subsection 2.5]{Vogtmann}. We refer the reader to this source for more details. For a proof of the statement, one can look at \cite{Servatius:generatorssource1} and \cite{Laurence:generatorssource2}, which are the original sources of this theorem. Servatius conjectured it and proved it for some special cases and Laurence gave a proof for general graphs. Note that the terminology in \cite{Servatius:generatorssource1} and \cite{Laurence:generatorssource2} is slightly different to the one we use here.\par
For simplicity, we assume that the graph $\Gamma$ consists of vertices $v_1,\dots,v_n$. There are four types of automorphisms that together generate $\Aut(A_\Gamma)$.
\begin{itemize}
    \item \textit{$\Gamma$-legal transvections:} For vertices $v_i$ and $v_j$ that satisfy the condition \mbox{$\lk(v_i) \subseteq \st(v_j)$}, we define the automorphism $T^l_{i,j}$ of $A_\Gamma$ by mapping the generators as follows: $v_i \mapsto v_jv_i$ and $v_k \mapsto v_k$ for $k \neq i$. These are called $\Gamma$-legal (left) transvections. Analogously, again for $v_i$ and $v_j$ with $\lk(v_i) \subseteq~\st(v_j)$, one can define the $\Gamma$-legal right transvection $T^r_{i,j}$ by mapping $v_i \mapsto~v_iv_j$ and $v_k \mapsto v_k$ for $k \neq i$.
    \item \textit{Partial conjugations:} Another type of generators are the so-called ($\Gamma$-legal) partial conjugations. Here, we often omit the term $\Gamma$-legal as we do not introduce other partial conjugations. For a vertex $v_j \in V(\Gamma)$ and a component $A$ of $\Gamma - \st(v_j)$, we define the partial conjugation $P_j^A$ by $v_i \mapsto v_jv_iv_j^{-1}$ for $v_i \in A$ and $v_k \mapsto v_k$ for $v_k \notin A$.
    \item \textit{$\Gamma$-legal permutations:} In order to get another automorphism of $A_\Gamma$, we can permute the generators. However, not all permutations give automorphisms, but only those that correspond to automorphisms of the graph $\Gamma$. We call these $\Gamma$-legal permutations. Due to the fact that these automorphisms of $A_\Gamma$ are related to graph automorphisms of $\Gamma$, $\Gamma$-legal permutations are often also called ``graph automorphisms''.
    \item \textit{Inversions:} For $j \in \{1,\dots,n\}$, we define the inversion $I_j$ as follows: \mbox{$v_j \mapsto v_j^{-1}$} and $v_k \mapsto v_k$ for $k \neq j$.
\end{itemize}
\begin{rem}
In a general setting, we write $T^l_{v_i,v_j}$ instead of $T^l_{i,j}$. The same applies to the notation for the other types of generators.
\end{rem}
We use the notation from above also for the images of the generators in $\Out(A_\Gamma)$. These elements then generate $\Out(A_\Gamma)$.

\subsection{Pure symmetric outer automorphisms}
The following is based on \cite[Section 2.2]{DayWade}. An automorphism $\phi$ of $A_\Gamma$ is called pure symmetric if every generator $v_i$ is mapped to a conjugate of itself. Note that the conjugating element may depend on $v_i$. The set of such elements forms a subgroup of $\Aut(A_\Gamma)$ as the composition of two pure symmetric automorphisms is again pure symmetric. This subgroup is called the group of pure symmetric automorphisms and denoted by $\PSA(A_\Gamma)$. A generating set for it is the set of partial conjugations; see \mbox{\cite[Theorem 2.5]{DayWade}}. We define $\PSO(A_\Gamma)$, the group of pure symmetric outer automorphisms, as the image of $\PSA(A_\Gamma)$ in $\Out(A_\Gamma)$.\par
Our goal is to first find conditions for when $\PSO(A_\Gamma)$ has finite-index in $\Out(A_\Gamma)$ and for when it is a right-angled Artin group. Then we show in Section \ref{chap:construction} that every right-angled Artin group $A_\Lambda$ occurs as $\PSO(A_\Gamma)$ for some graph $\Gamma$ for which $\PSO(A_\Gamma)$ has finite index in $\Out(A_\Gamma)$.

\subsubsection{Finite-index condition for \texorpdfstring{$\PSO(A_\Gamma)$}{PSO(RAAG)}}\label{sec:finiteindexconditionforPSO(RAAG)}
The following theorem is based on \cite[Appendix A]{WadeBrück}.
\begin{thm}\label{thm:condforfiniteindex}
Let $\Gamma$ be a graph. Then the condition
\begin{equation}\label{eq:conditionforfiniteindex}
    \forall v,w \in V(\Gamma): \: \lk(v) \subseteq \st(w) \Longrightarrow v = w
\end{equation}
is equivalent to the group $\PSO(A_\Gamma)$ having finite index in $\Out(A_\Gamma)$.
\end{thm}
Note that condition \eqref{eq:conditionforfiniteindex} is equivalent to the group $A_\Gamma$ having no $\Gamma$-legal transvections. The idea for the proof that condition \eqref{eq:conditionforfiniteindex} implies finite index is to show that
\[\Out(A_\Gamma)/\PSO(A_\Gamma) \cong \Inv \rtimes \Per,\]
where $\Inv$ and $\Per$ denote the subgroups of $\Out(A_\Gamma)$ generated by the inversions respectively by the $\Gamma$-legal permutations. For the other direction, one can show that the existence of a $\Gamma$-legal transvection implies that the quotient has infinitely many elements.

\subsubsection{Condition for when \texorpdfstring{$\PSO(A_\Gamma)$}{PSO(RAAG)} is a right-angled Artin group}
In this subsection, we state a condition about when $\PSO(A_\Gamma)$ is a right-angled Artin group. This condition was developed and proved by Day and Wade in \cite{DayWade}. This is also the source for this subsection, in particular \cite[Chapters 2 and 5]{DayWade}. We first need the following two definitions.
\begin{defn}
Let $\Gamma$ be a graph and $v \neq w \in V(\Gamma)$. The pair $(v,w)$ is called a separating intersection of links if $v$ is not adjacent to $w$, i.e. $\{v,w\} \notin E(\Gamma)$, and $\Gamma - (\lk(v) \cap \lk(w))$ has a component that contains neither $v$ nor $w$.
\end{defn}
We often use the formulation ``$(v,w)$ is a SIL-pair'' for a pair $(v,w)$ that is a separating intersection of links.
\begin{defn}\label{def:suppgraph}
Let $\Gamma$ be a graph and $v \in V(\Gamma)$ be a vertex of $\Gamma$. We define the support graph $\sg{\Gamma}{v}$ as follows. For every component $C$ of $\Gamma - \st(v)$ there is a vertex in $\sg{\Gamma}{v}$. Two vertices $A$ and $B$ in $\sg{\Gamma}{v}$ are connected by an edge if there is a vertex $b \in B$ such that $A$ is also a component of $\Gamma - \st(b)$.
\end{defn}
\begin{rem}
Note that the definition above is not symmetric. That is, for two components $A$ and $B$ of $\Gamma - \st(v)$ the fact that there is a $b \in B$ such that $A$ is also a component of $\Gamma - \st(b)$ is not equivalent to the fact that there is an $a \in A$ such that $B$ is also a component of $\Gamma - \st(a)$. There is an edge between $A$ and $B$ in the support graph if any of these two conditions holds.
\end{rem}
\begin{rem}
As in Definition \ref{def:suppgraph}, we often use the same symbol for the vertices in the support graph $\sg{\Gamma}{v}$ and the components of $\Gamma - \st(v)$, even though these are not the same.
\end{rem}
For $b \in \Gamma - \st(v)$, we often write the component of $\Gamma - \st(v)$ that contains $b$ as $[b]_v$. So, equivalently there is an edge between two vertices $A$ and $B$ in the support graph $\sg{\Gamma}{v}$ if there is a $b \in \Gamma - \st(v)$ such that $B = [b]_v$ and $A$ is also a component of $\Gamma - \st(b)$.\par
We now define the following graph $\Theta$ (depending on $\Gamma$). This graph has two types of vertices. We call them vertices of type \rom{1} and \rom{2}. More precisely, for every vertex $v \in V(\Gamma)$ we have the following vertices.
\begin{itemize}
    \item Vertices of type \rom{1}: For every edge $e$ in $\sg{\Gamma}{v}$ we have a vertex $\alpha_e^v$.
    \item Vertices of type \rom{2}: We also have a vertex for every component of $\sg{\Gamma}{v}$ except one, i.e. we have vertices $\beta_1^v, \dots, \beta_{\min(N(v)-1,0)}^v$, where $N(v)$ is the number of components of $\sg{\Gamma}{v}$. We take the minimum with 0 to avoid the special case $N(v) = 0$. This happens if $\Gamma - \st(v)$ has no vertices and thus also $\sg{\Gamma}{v}$ has no vertices.
\end{itemize}
Concerning the edges, vertices of type \rom{2} are connected to all other vertices and vertices $\alpha_e^v$ and $\alpha_f^w$ of type \rom{1} are connected except when $(v,w)$ is a SIL-pair and the edges are of the form $e = \{[w]_v,L\}$ and $f = \{[v]_w,L\}$, where $L$ is a component of both $\Gamma - \st(v)$ and $\Gamma - \st(w)$.\par
We now state the theorem about when $\PSO(A_\Gamma)$ is a right-angled Artin group.
\begin{thm}[{\cite[Theorem 5.12]{DayWade}}]\label{thm:condforpsobeingaraag}
Let $\Gamma$ be a graph. Then $\PSO(A_\Gamma)$ is a right-angled Artin group if and only if all support graphs of $\Gamma$ are forests. In this case, $\PSO(A_\Gamma) \cong A_\Theta$ for $\Theta$ as defined above.
\end{thm}
\begin{rem}
Note that in this theorem, in contrast to Definition \ref{def:raag}, we also treat the group $\{\id\} = A_\emptyset$ as a right-angled Artin group, where $\emptyset$ denotes the graph with no vertices. Namely, $\Theta$ is the empty graph if all support graphs have at most one vertex. This happens for example when $\Gamma$ is a complete graph. Note that it makes sense that $\PSO(A_\Gamma)$ is the trivial group in this case since when $\Gamma$ is complete, then $A_\Gamma$ is $\Z^{|V(\Gamma)|}$. Thus, conjugation by any element does nothing, so there are no pure symmetric outer automorphisms except the identity.
\end{rem}
In \cite[Chapter 5]{DayWade}, one can find an explicit construction of the isomorphism between the right-angled Artin group $A_\Theta$ and the group $\PSO(A_\Gamma)$ in case all support graphs of $\Gamma$ are forests. They define the vertices of type \rom{2} more precisely in \cite[Definition 5.4]{DayWade}, describe the needed generators of $\PSO(A_\Gamma)$ in \cite[Section 5.1]{DayWade} and give the correspondence between the generators of $A_\Theta$ and the generators of $\PSO(A_\Gamma)$ in \cite[Proposition 5.5]{DayWade}.

\section{Right-angled Artin groups as finite-index subgroups of \texorpdfstring{$\Out(A_\Gamma)$}{Out(RAAG)}}\label{chap:construction}
In this section, we use Theorems \ref{thm:condforfiniteindex} and \ref{thm:condforpsobeingaraag} to show that every right-angled Artin group is a finite-index subgroup of the outer automorphism group of some other right-angled Artin group. This is Theorem \ref{thm:anyraagisfinindsubgrofoutraagintro}, which we here restate as Theorem \ref{thm:anyraagisfinindsubgrofoutraag}.
\begin{thm}\label{thm:anyraagisfinindsubgrofoutraag}
For any graph $\Lambda$, there is a graph $\Gamma = \Gamma(\Lambda)$ such that
\begin{enumerate}
    \item $A_\Lambda \cong \PSO(A_\Gamma)$ and
    \item $\PSO(A_\Gamma)$ has finite index in $\Out(A_\Gamma)$.
\end{enumerate}
\end{thm}
\begin{rem}
We first want to comment on how this theorem and also Theorem \ref{thm:anyraagisfinindsubgrofoutraagnographauto} below were developed. We used computer programs that can be found on \cite{Code}. Using these, we found out that for any graph with at most four vertices the corresponding right-angled Artin group occurs as $\PSO(A_\Gamma)$ for some graph $\Gamma$. We then generalised these examples step by step until we arrived at the constructions given below. These computer programs are based on and contain parts of programs written by Benjamin Brück, which can be found on \cite{Code:Benjamin} and were used for \cite{WadeBrück}.
\end{rem}
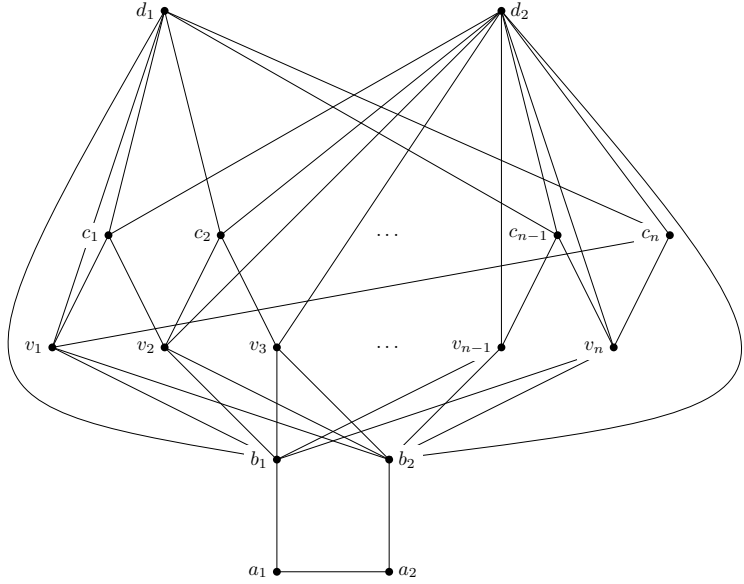
\begin{figure}[ht]
    \centering
    \resizebox{\textwidth}{!}{
    \centering
    \begin{tikzpicture}
    
    \coordinate (v1) at (0,0);
    \coordinate (v2) at (2,0);
    \coordinate (v3) at (4,0);
    \coordinate[label=center:\dots] (Dots1) at (6,0);
    \coordinate (vn-1) at (8,0);
    \coordinate (vn) at (10,0);
    \coordinate (C1) at (1,2);
    \coordinate (C2) at (3,2);
    \coordinate[label=center:\dots] (Dots2) at (6,2);
    \coordinate (Cn-1) at (9,2);
    \coordinate (Cn) at (11,2);
    \coordinate[label=left: $d_1$] (D1) at (2,6);
    \coordinate[label=right: $d_2$] (D2) at (8,6);
    \coordinate (B1) at (4,-2);
    \coordinate (B2) at (6,-2);
    \coordinate[label=left: $a_1$] (A1) at (4,-4);
    \coordinate[label=right: $a_2$] (A2) at (6,-4);

    \node at (D1) [left = 0.4mm of D1, fill=white] {$d_1$};
    \node at (D2) [right = 0.4mm of D2, fill=white] {$d_2$};

    \draw (v1) -- (D1);
    \draw (v1) -- (Cn);
    \draw (vn-1) -- (B1);
    \draw[] (B1) .. controls(-2,-1) .. (D1);
    \draw (vn) -- (B1);
    \draw (vn) -- (B2);
    \draw[] (B2) .. controls(14,-1) .. (D2);
    \draw (v1) -- (B1);
    \draw (Cn) -- (D1);

    \node at (vn-1) [left = 0.4mm of vn-1, fill=white] {$v_{n-1}$};

    \draw (vn-1) -- (B2);

    \node at (v1) [left = 0.4mm of v1, fill=white] {$v_1$};
    \node at (v2) [left = 0.4mm of v2, fill=white] {$v_2$};
    \node at (v3) [left = 0.4mm of v3, fill=white] {$v_3$};
    \node at (vn) [left = 0.4mm of vn, fill=white] {$v_n$};
    \node at (C1) [left = 0.4mm of C1, fill=white] {$c_1$};
    \node at (C2) [left = 0.4mm of C2, fill=white] {$c_2$};
    \node at (Cn-1) [left = 0.4mm of Cn-1, fill=white] {$c_{n-1}$};
    \node at (Cn) [left = 0.4mm of Cn, fill=white] {$c_{n}$};
    \node at (B1) [left = 0.4mm of B1, fill=white] {$b_1$};
    \node at (B2) [right = 0.4mm of B2, fill = white] {$b_2$};
    \node at (A1) [left = 0.4mm of A1, fill=white] {$a_1$};
    \node at (A2) [right = 0.4mm of A2, fill = white] {$a_2$};

    \draw (v1) -- (C1);
    \draw (v2) -- (C1);
    \draw (v2) -- (C2);
    \draw (v3) -- (C2);
    \draw (vn-1) -- (Cn-1);
    \draw (vn) -- (Cn-1);
    \draw (vn) -- (Cn);
    \draw (v2) -- (D2);
    \draw (v3) -- (D2);
    \draw (vn-1) -- (D2);
    \draw (vn) -- (D2);
    \draw (C1) -- (D1);
    \draw (C2) -- (D1);
    \draw (Cn-1) -- (D1);
    \draw (C1) -- (D2);
    \draw (C2) -- (D2);
    \draw (Cn-1) -- (D2);
    \draw (Cn) -- (D2);
    \draw (v2) -- (B1);
    \draw (v3) -- (B1);
    \draw (v1) -- (B2);
    \draw (v2) -- (B2);
    \draw (v3) -- (B2);
    \draw (B1) -- (A1);
    \draw (B2) -- (A2);
    \draw (A1) -- (A2);

    \fill (v1) circle (2pt);
    \fill (v2) circle (2pt);
    \fill (v3) circle (2pt);
    \fill (vn-1) circle (2pt);
    \fill (vn) circle (2pt);
    \fill (C1) circle (2pt);
    \fill (C2) circle (2pt);
    \fill (Cn-1) circle (2pt);
    \fill (Cn) circle (2pt);
    \fill (D1) circle (2pt);
    \fill (D2) circle (2pt);
    \fill (B1) circle (2pt);
    \fill (B2) circle (2pt);
    \fill (A1) circle (2pt);
    \fill (A2) circle (2pt);
    \end{tikzpicture}
    }
    \caption{Construction of $\Gamma$}
    \label{fig:construction}
\end{figure}
We assume without loss of generality that \mbox{$V(\Lambda) = \{v_1, \dots, v_n\}$}. We also assume $n \geq 3$. The other cases are covered in Appendix \ref{app:smallgraphs}. For a given graph $\Lambda$, we define the graph $\Gamma = \Gamma(\Lambda)$ as follows. We obtain $\Gamma$ from $\Lambda$ by adding vertices $a_1$, $a_2$, $b_1$, $b_2$, $c_1$, \dots, $c_n$, $d_1$ and $d_2$ and the edges depicted in Figure \ref{fig:construction}: The vertex $d_1$ is connected to $b_1$, to all vertices $c_i$ and to the vertex $v_1$, but not to the vertices $v_j$ for $j>1$. The vertex $d_2$ is connected to $b_2$, all $c_i$ and all $v_j$ except $v_1$. The vertices $c_i$ are connected to $d_1$, $d_2$ and the vertices $v_i$ and $v_{i+1}$, where $v_{n+1} \coloneqq v_1$. In addition to the already defined edges, the vertices $v_j$ are connected to $b_1$ and $b_2$. Furthermore, we have the edges $\{b_1, a_1\}$, $\{b_2, a_2\}$ and $\{a_1, a_2\}$. Finally, if there are edges in $\Lambda$, these are also present in $\Gamma$. But as we work with an arbitrary graph $\Lambda$, we did not draw them in Figure \ref{fig:construction}.\par
We want to prove the following two lemmas, which together imply Theorem \ref{thm:anyraagisfinindsubgrofoutraag}.\hspace*{-0.1cm}
\begin{lem}\label{lem:PSOAGammaisALambda}
For this graph $\Gamma = \Gamma(\Lambda)$, we have that $\PSO(A_\Gamma)$ is isomorphic to $A_\Lambda$.
\end{lem}
\begin{lem}\label{lem:PSOAGammahasfiniteindex}
For this graph $\Gamma = \Gamma(\Lambda)$, $\PSO(A_\Gamma)$ has finite index in $\Out(A_\Gamma)$.
\end{lem}
\begin{proof}[Proof of Lemma \ref{lem:PSOAGammaisALambda}]
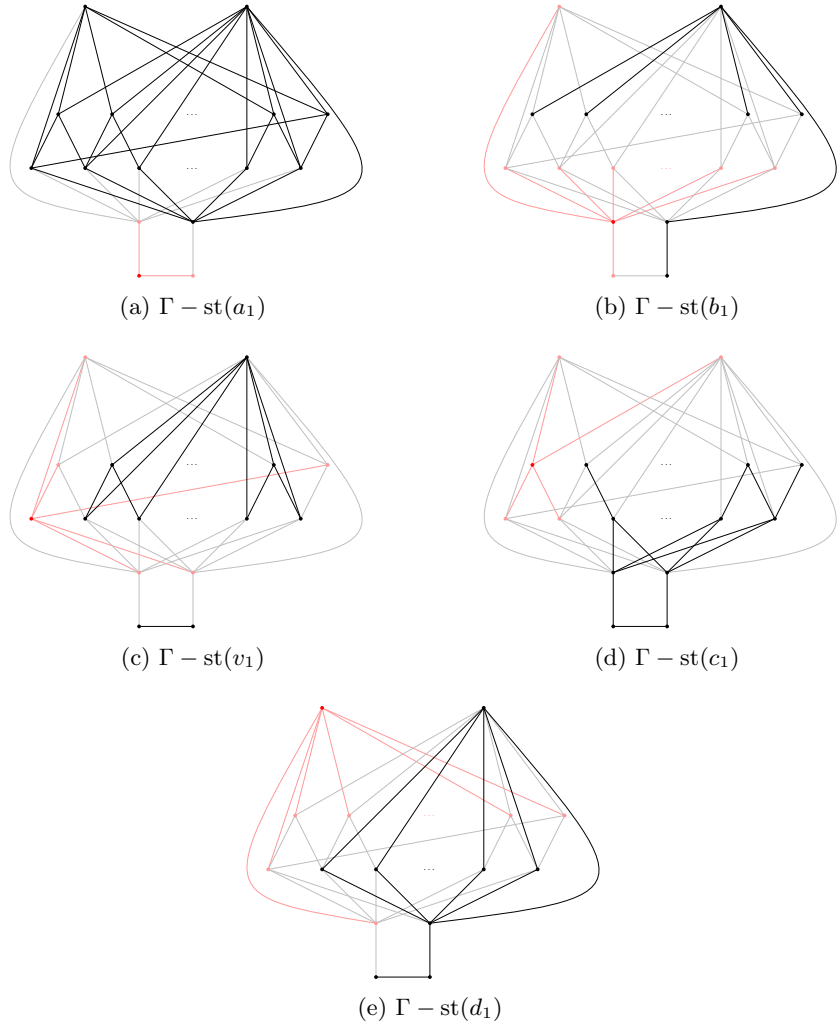
\begin{figure}
    \centering
    \begin{subfigure}[b]{0.48\textwidth}
        \centering
        \resizebox{\textwidth}{!}{
        \centering
        \begin{tikzpicture}
            \coordinate[] (v1) at (0,0);
            \coordinate[] (v2) at (2,0);
            \coordinate[] (v3) at (4,0);
            \coordinate[label=center:\dots] (Dots1) at (6,0);
            \coordinate[] (vn-1) at (8,0);
            \coordinate[] (vn) at (10,0);
            \coordinate[] (C1) at (1,2);
            \coordinate[] (C2) at (3,2);
            \coordinate[label=center:\dots] (Dots2) at (6,2);
            \coordinate[] (Cn-1) at (9,2);
            \coordinate[] (Cn) at (11,2);
            \coordinate[] (D1) at (2,6);
            \coordinate[] (D2) at (8,6);
            \coordinate[] (B1) at (4,-2);
            \coordinate[] (B2) at (6,-2);
            \coordinate[] (A1) at (4,-4);
            \coordinate[] (A2) at (6,-4);
            \draw[lightred] (B1) -- (A1);
            \draw[lightred] (A1) -- (A2);
            \draw[lightgray] (v1) -- (B1);
            \draw[lightgray] (v2) -- (B1);
            \draw[lightgray] (v3) -- (B1);
            \draw[lightgray] (vn-1) -- (B1);
            \draw[lightgray] (vn) -- (B1);
            \draw[lightgray] (B2) -- (A2);
            \draw[lightgray] (B1) .. controls(-2,-1) .. (D1);
            \draw (v1) -- (C1);
            \draw (v2) -- (C1);
            \draw (v2) -- (C2);
            \draw (v3) -- (C2);
            \draw (vn-1) -- (Cn-1);
            \draw (vn) -- (Cn-1);
            \draw (vn) -- (Cn);
            \draw (v1) -- (Cn);
            \draw (v1) -- (D1);
            \draw (vn) -- (D2);
            \draw (v2) -- (D2);
            \draw (v3) -- (D2);
            \draw (vn-1) -- (D2);
            \draw (C1) -- (D1);
            \draw (C2) -- (D1);
            \draw (Cn-1) -- (D1);
            \draw (Cn) -- (D1);
            \draw (C1) -- (D2);
            \draw (C2) -- (D2);
            \draw (Cn-1) -- (D2);
            \draw (Cn) -- (D2);
            \draw (v1) -- (B2);
            \draw (v2) -- (B2);
            \draw (v3) -- (B2);
            \draw (vn-1) -- (B2);
            \draw (vn) -- (B2);
            \draw[] (B2) .. controls(14,-1) .. (D2);
            \fill (v1) circle (2pt);
            \fill (v2) circle (2pt);
            \fill (v3) circle (2pt);
            \fill (vn-1) circle (2pt);
            \fill (vn) circle (2pt);
            \fill (C1) circle (2pt);
            \fill (C2) circle (2pt);
            \fill (Cn-1) circle (2pt);
            \fill (Cn) circle (2pt);
            \fill (D1) circle (2pt);
            \fill (D2) circle (2pt);
            \fill[lightred] (B1) circle (2pt);
            \fill (B2) circle (2pt);
            \fill[red] (A1) circle (2pt);
            \fill[lightred] (A2) circle (2pt);
        \end{tikzpicture}
        }
        \caption{$\Gamma - \st(a_1)$}
        \label{fig:GammaminusstA1}
    \end{subfigure}
    \hfill
    \begin{subfigure}[b]{0.48\textwidth}
        \centering
        \resizebox{\textwidth}{!}{
        \centering
        \begin{tikzpicture}
            \coordinate[] (v1) at (0,0);
            \coordinate[] (v2) at (2,0);
            \coordinate[] (v3) at (4,0);
            \coordinate[label=center:{\color{lightred}\dots}] (Dots1) at (6,0);
            \coordinate[] (vn-1) at (8,0);
            \coordinate[] (vn) at (10,0);
            \coordinate[] (C1) at (1,2);
            \coordinate[] (C2) at (3,2);
            \coordinate[label=center:\dots] (Dots2) at (6,2);
            \coordinate[] (Cn-1) at (9,2);
            \coordinate[] (Cn) at (11,2);
            \coordinate[] (D1) at (2,6);
            \coordinate[] (D2) at (8,6);
            \coordinate[] (B1) at (4,-2);
            \coordinate[] (B2) at (6,-2);
            \coordinate[] (A1) at (4,-4);
            \coordinate[] (A2) at (6,-4);
            \draw[lightred] (v1) -- (B1);
            \draw[lightred] (v2) -- (B1);
            \draw[lightred] (v3) -- (B1);
            \draw[lightred] (vn-1) -- (B1);
            \draw[lightred] (vn) -- (B1);
            \draw[lightred] (B1) -- (A1);
            \draw[lightred] (B1) .. controls(-2,-1) .. (D1);
            \draw[lightgray] (v1) -- (C1);
            \draw[lightgray] (v2) -- (C1);
            \draw[lightgray] (v2) -- (C2);
            \draw[lightgray] (v3) -- (C2);
            \draw[lightgray] (vn-1) -- (Cn-1);
            \draw[lightgray] (vn) -- (Cn-1);
            \draw[lightgray] (vn) -- (Cn);
            \draw[lightgray] (v1) -- (Cn);
            \draw[lightgray] (v1) -- (D1);
            \draw[lightgray] (vn) -- (D2);
            \draw[lightgray] (v2) -- (D2);
            \draw[lightgray] (v3) -- (D2);
            \draw[lightgray] (vn-1) -- (D2);
            \draw[lightgray] (C1) -- (D1);
            \draw[lightgray] (C2) -- (D1);
            \draw[lightgray] (Cn-1) -- (D1);
            \draw[lightgray] (Cn) -- (D1);
            \draw[lightgray] (v1) -- (B2);
            \draw[lightgray] (v2) -- (B2);
            \draw[lightgray] (v3) -- (B2);
            \draw[lightgray] (vn-1) -- (B2);
            \draw[lightgray] (vn) -- (B2);
            \draw[lightgray] (A1) -- (A2);
            \draw (C1) -- (D2);
            \draw (C2) -- (D2);
            \draw (Cn-1) -- (D2);
            \draw (Cn) -- (D2);
            \draw (B2) -- (A2);
            \draw[] (B2) .. controls(14,-1) .. (D2);
            \fill[lightred] (v1) circle (2pt);
            \fill[lightred] (v2) circle (2pt);
            \fill[lightred] (v3) circle (2pt);
            \fill[lightred] (vn-1) circle (2pt);
            \fill[lightred] (vn) circle (2pt);
            \fill (C1) circle (2pt);
            \fill (C2) circle (2pt);
            \fill (Cn-1) circle (2pt);
            \fill (Cn) circle (2pt);
            \fill[lightred] (D1) circle (2pt);
            \fill (D2) circle (2pt);
            \fill[red] (B1) circle (2pt);
            \fill (B2) circle (2pt);
            \fill[lightred] (A1) circle (2pt);
            \fill (A2) circle (2pt);
        \end{tikzpicture}
        }
        \caption{$\Gamma - \st(b_1)$}
        \label{fig:GammaminusstB1}
    \end{subfigure}\\
    \vspace{0.4cm}
    \begin{subfigure}[b]{0.48\textwidth}
        \centering
        \resizebox{\textwidth}{!}{
        \centering
        \begin{tikzpicture}
            \coordinate[] (v1) at (0,0);
            \coordinate[] (v2) at (2,0);
            \coordinate[] (v3) at (4,0);
            \coordinate[label=center:\dots] (Dots1) at (6,0);
            \coordinate[] (vn-1) at (8,0);
            \coordinate[] (vn) at (10,0);
            \coordinate[] (C1) at (1,2);
            \coordinate[] (C2) at (3,2);
            \coordinate[label=center:\dots] (Dots2) at (6,2);
            \coordinate[] (Cn-1) at (9,2);
            \coordinate[] (Cn) at (11,2);
            \coordinate[] (D1) at (2,6);
            \coordinate[] (D2) at (8,6);
            \coordinate[] (B1) at (4,-2);
            \coordinate[] (B2) at (6,-2);
            \coordinate[] (A1) at (4,-4);
            \coordinate[] (A2) at (6,-4);
            \draw[lightred] (v1) -- (C1);
            \draw[lightred] (v1) -- (Cn);
            \draw[lightred] (v1) -- (D1);
            \draw[lightred] (v1) -- (B1);
            \draw[lightred] (v1) -- (B2);
            \draw[lightgray] (v2) -- (C1);
            \draw[lightgray] (vn) -- (Cn);
            \draw[lightgray] (C1) -- (D1);
            \draw[lightgray] (C2) -- (D1);
            \draw[lightgray] (Cn-1) -- (D1);
            \draw[lightgray] (Cn) -- (D1);
            \draw[lightgray] (C1) -- (D2);
            \draw[lightgray] (Cn) -- (D2);
            \draw[lightgray] (v2) -- (B1);
            \draw[lightgray] (v3) -- (B1);
            \draw[lightgray] (vn) -- (B1);
            \draw[lightgray] (vn-1) -- (B1);
            \draw[lightgray] (v2) -- (B2);
            \draw[lightgray] (v3) -- (B2);
            \draw[lightgray] (vn) -- (B2);
            \draw[lightgray] (vn-1) -- (B2);
            \draw[lightgray] (B1) -- (A1);
            \draw[lightgray] (B2) -- (A2);
            \draw[lightgray] (B1) .. controls(-2,-1) .. (D1);
            \draw[lightgray] (B2) .. controls(14,-1) .. (D2);
            \draw (v2) -- (C2);
            \draw (v3) -- (C2);
            \draw (vn-1) -- (Cn-1);
            \draw (vn) -- (Cn-1);
            \draw (vn) -- (D2);
            \draw (v2) -- (D2);
            \draw (v3) -- (D2);
            \draw (vn-1) -- (D2);
            \draw (C2) -- (D2);
            \draw (Cn-1) -- (D2);
            \draw (A1) -- (A2);
            \fill[red] (v1) circle (2pt);
            \fill (v2) circle (2pt);
            \fill (v3) circle (2pt);
            \fill (vn-1) circle (2pt);
            \fill (vn) circle (2pt);
            \fill[lightred] (C1) circle (2pt);
            \fill (C2) circle (2pt);
            \fill (Cn-1) circle (2pt);
            \fill[lightred] (Cn) circle (2pt);
            \fill[lightred] (D1) circle (2pt);
            \fill (D2) circle (2pt);
            \fill[lightred] (B1) circle (2pt);
            \fill[lightred] (B2) circle (2pt);
            \fill (A1) circle (2pt);
            \fill (A2) circle (2pt);
        \end{tikzpicture}
        }
        \caption{$\Gamma - \st(v_1)$}
        \label{fig:Gammaminusstv1}
    \end{subfigure}
    \hfill
    \begin{subfigure}[b]{0.48\textwidth}
        \centering
        \resizebox{\textwidth}{!}{
        \centering
        \begin{tikzpicture}
            \coordinate[] (v1) at (0,0);
            \coordinate[] (v2) at (2,0);
            \coordinate[] (v3) at (4,0);
            \coordinate[label=center:\dots] (Dots1) at (6,0);
            \coordinate[] (vn-1) at (8,0);
            \coordinate[] (vn) at (10,0);
            \coordinate[] (C1) at (1,2);
            \coordinate[] (C2) at (3,2);
            \coordinate[label=center:\dots] (Dots2) at (6,2);
            \coordinate[] (Cn-1) at (9,2);
            \coordinate[] (Cn) at (11,2);
            \coordinate[] (D1) at (2,6);
            \coordinate[] (D2) at (8,6);
            \coordinate[] (B1) at (4,-2);
            \coordinate[] (B2) at (6,-2);
            \coordinate[] (A1) at (4,-4);
            \coordinate[] (A2) at (6,-4);
            \draw[lightred] (v1) -- (C1);
            \draw[lightred] (v2) -- (C1);
            \draw[lightred] (C1) -- (D1);
            \draw[lightred] (C1) -- (D2);
            \draw[lightgray] (v2) -- (C2);
            \draw[lightgray] (v1) -- (Cn);
            \draw[lightgray] (v1) -- (D1);
            \draw[lightgray] (vn) -- (D2);
            \draw[lightgray] (v2) -- (D2);
            \draw[lightgray] (v3) -- (D2);
            \draw[lightgray] (vn-1) -- (D2);
            \draw[lightgray] (C2) -- (D1);
            \draw[lightgray] (Cn-1) -- (D1);
            \draw[lightgray] (Cn) -- (D1);
            \draw[lightgray] (C2) -- (D2);
            \draw[lightgray] (Cn-1) -- (D2);
            \draw[lightgray] (Cn) -- (D2);
            \draw[lightgray] (v1) -- (B1);
            \draw[lightgray] (v2) -- (B1);
            \draw[lightgray] (v1) -- (B2);
            \draw[lightgray] (v2) -- (B2);
            \draw[lightgray] (B1) .. controls(-2,-1) .. (D1);
            \draw[lightgray] (B2) .. controls(14,-1) .. (D2);
            \draw (v3) -- (C2);
            \draw (vn-1) -- (Cn-1);
            \draw (vn) -- (Cn-1);
            \draw (vn) -- (Cn);
            \draw (v3) -- (B1);
            \draw (vn) -- (B1);
            \draw (vn-1) -- (B1);
            \draw (v3) -- (B2);
            \draw (vn) -- (B2);
            \draw (vn-1) -- (B2);
            \draw (B1) -- (A1);
            \draw (B2) -- (A2);
            \draw (A1) -- (A2);
            \fill[lightred] (v1) circle (2pt);
            \fill[lightred] (v2) circle (2pt);
            \fill (v3) circle (2pt);
            \fill (vn-1) circle (2pt);
            \fill (vn) circle (2pt);
            \fill[red] (C1) circle (2pt);
            \fill (C2) circle (2pt);
            \fill (Cn-1) circle (2pt);
            \fill (Cn) circle (2pt);
            \fill[lightred] (D1) circle (2pt);
            \fill[lightred] (D2) circle (2pt);
            \fill (B1) circle (2pt);
            \fill (B2) circle (2pt);
            \fill (A1) circle (2pt);
            \fill (A2) circle (2pt);
        \end{tikzpicture}
        }
        \caption{$\Gamma - \st(c_1)$}
        \label{fig:GammaminusstC1}
    \end{subfigure}\\
    \vspace{0.4cm}
    \begin{subfigure}[b]{0.48\textwidth}
        \centering
        \resizebox{\textwidth}{!}{
        \centering
        \begin{tikzpicture}
            \coordinate[] (v1) at (0,0);
            \coordinate[] (v2) at (2,0);
            \coordinate[] (v3) at (4,0);
            \coordinate[label=center:\dots] (Dots1) at (6,0);
            \coordinate[] (vn-1) at (8,0);
            \coordinate[] (vn) at (10,0);
            \coordinate[] (C1) at (1,2);
            \coordinate[] (C2) at (3,2);
            \coordinate[label=center:{\color{lightred}\dots}] (Dots2) at (6,2);
            \coordinate[] (Cn-1) at (9,2);
            \coordinate[] (Cn) at (11,2);
            \coordinate[] (D1) at (2,6);
            \coordinate[] (D2) at (8,6);
            \coordinate[] (B1) at (4,-2);
            \coordinate[] (B2) at (6,-2);
            \coordinate[] (A1) at (4,-4);
            \coordinate[] (A2) at (6,-4);
            \draw[lightred] (v1) -- (D1);
            \draw[lightred] (C1) -- (D1);
            \draw[lightred] (C2) -- (D1);
            \draw[lightred] (Cn-1) -- (D1);
            \draw[lightred] (Cn) -- (D1);
            \draw[lightred] (B1) .. controls(-2,-1) .. (D1);
            \draw[lightgray] (v1) -- (C1);
            \draw[lightgray] (v2) -- (C1);
            \draw[lightgray] (v2) -- (C2);
            \draw[lightgray] (v3) -- (C2);
            \draw[lightgray] (vn-1) -- (Cn-1);
            \draw[lightgray] (vn) -- (Cn-1);
            \draw[lightgray] (vn) -- (Cn);
            \draw[lightgray] (v1) -- (Cn);
            \draw[lightgray] (C1) -- (D2);
            \draw[lightgray] (C2) -- (D2);
            \draw[lightgray] (Cn-1) -- (D2);
            \draw[lightgray] (Cn) -- (D2);
            \draw[lightgray] (v1) -- (B1);
            \draw[lightgray] (v2) -- (B1);
            \draw[lightgray] (v3) -- (B1);
            \draw[lightgray] (vn-1) -- (B1);
            \draw[lightgray] (vn) -- (B1);
            \draw[lightgray] (v1) -- (B2);
            \draw[lightgray] (B1) -- (A1);
            \draw (vn) -- (D2);
            \draw (v2) -- (D2);
            \draw (v3) -- (D2);
            \draw (vn-1) -- (D2);
            \draw (v2) -- (B2);
            \draw (v3) -- (B2);
            \draw (vn-1) -- (B2);
            \draw (vn) -- (B2);
            \draw (B2) -- (A2);
            \draw (A1) -- (A2);
            \draw[] (B2) .. controls(14,-1) .. (D2);
            \fill[lightred] (v1) circle (2pt);
            \fill (v2) circle (2pt);
            \fill (v3) circle (2pt);
            \fill (vn-1) circle (2pt);
            \fill (vn) circle (2pt);
            \fill[lightred] (C1) circle (2pt);
            \fill[lightred] (C2) circle (2pt);
            \fill[lightred] (Cn-1) circle (2pt);
            \fill[lightred] (Cn) circle (2pt);
            \fill[red] (D1) circle (2pt);
            \fill (D2) circle (2pt);
            \fill[lightred] (B1) circle (2pt);
            \fill (B2) circle (2pt);
            \fill (A1) circle (2pt);
            \fill (A2) circle (2pt);
        \end{tikzpicture}
        }
        \caption{$\Gamma - \st(d_1)$}
        \label{fig:GammaminusstD1}
    \end{subfigure}
    \caption{$\Gamma - \st(w)$ for some important vertices $w \in V(\Gamma)$}
    \label{fig:Gammaminusstwforseveralw}
\end{figure}
We want to use Theorem \ref{thm:condforpsobeingaraag}. So, we first need to show that every support graph is a forest. To compute the support graphs, we need to consider $\Gamma - \st(w)$ for all vertices $w$ of $\Gamma$. Some of these are shown in Figure \ref{fig:Gammaminusstwforseveralw}. The vertex $w$ is coloured in red, the rest of $\st(w)$ is coloured in light red and all other deleted edges are coloured in grey. Note that for Figure \ref{fig:Gammaminusstv1}, it might be that also some of the vertices $\{v_2,\dots,v_n\}$ should be coloured light red, namely all vertices in $\st_\Lambda(v_1)$. For the vertices we did not consider in Figure \ref{fig:Gammaminusstwforseveralw}, $\Gamma - \st(w)$ is analogous to one of the vertices we considered there.\par
We can see that for all vertices $w$ that are not in $\Lambda$, $\Gamma - \st(w)$ is connected and thus the support graph $\sg{\Gamma}{w}$ consists of one vertex. Here, we used that $n \geq 3$, otherwise $\Gamma - \st(c_1)$ respectively $\Gamma - \st(c_2)$ have two components as $c_2$ respectively $c_1$ become isolated. For the vertices $v_i$, the graph $\Gamma - \st(v_i)$ has two components, one of which is $\{a_1,a_2\}$. The second component contains all other vertices that are not in the star of $v_i$. Using the definition of the support graph, this means that $\sg{\Gamma}{v_i}$ has two vertices. Whether there is an edge between them depends on whether $V(\Lambda) \setminus \st_\Lambda(v_i)$ is empty or not. If $V(\Lambda) \setminus \st_\Lambda(v_i) \neq \emptyset$, then there is a $v_j \in V(\Lambda) \setminus \st_\Lambda(v_i)$ and $\{a_1,a_2\}$ is also a component of $\Gamma - \st(v_j)$. Thus, by Definition \ref{def:suppgraph}, in this case there is an edge in the support graph. However, if $V(\Lambda) \setminus \st_\Lambda(v_i) = \emptyset$, we cannot satisfy the condition of Definition \ref{def:suppgraph}, so the support graph has no edge.\par
In summary, we see that all support graphs are forests and thus $\PSO(A_\Gamma)$ is a right-angled Artin group by Theorem \ref{thm:condforpsobeingaraag}. As in this theorem, we call the underlying graph for this right-angled Artin group $\Theta$. Furthermore, we know from the structure of the support graphs that there is exactly one vertex $u_i$ in the graph $\Theta$ for every vertex $v_i$ in the graph $\Gamma$. Depending on whether $V(\Lambda) \setminus \st_\Lambda(v_i) = \emptyset$ or not, it is a vertex of type \rom{2} or \rom{1}. If $V(\Lambda) \setminus \st_\Lambda(v_i) = \emptyset$, then
\[u_i \de \beta_1^{v_i}.\]
If $V(\Lambda) \setminus \st_\Lambda(v_i) \neq \emptyset$, then
\[u_i \de \alpha_{e_i}^{v_i},\]
where $e_i$ is the edge in the support graph $\sg{\Gamma}{v_i}$.\par
We claim that we have $v_i \sim v_j$ if and only if $u_i \sim u_j$, which shows that $\Lambda$ is isomorphic to $\Theta$. In order to show this claim, we distinguish whether $v_i$ is adjacent to $v_j$ or not.
\begin{description}[labelindent=\parindent,labelwidth = \widthof{$v_i \sim v_j$:},leftmargin=!]
    \item[$v_i \sim v_j$:] If $V(\Lambda) \setminus \st_\Lambda(v_i) = \emptyset$ or $V(\Lambda) \setminus \st_\Lambda(v_j) = \emptyset$, then $u_i \sim u_j$ as $u_i$ respectively $u_j$ is of type \rom{2} and thus connected to all other vertices. Otherwise, both are of type \rom{1}, but $(v_i,v_j)$ is not a SIL-pair as they are connected. Thus, we have $u_i \sim u_j$ as well.
    \item[$v_i \not\sim v_j$:] Here, both $V(\Lambda) \setminus \st_\Lambda(v_i)$ and $V(\Lambda) \setminus \st_\Lambda(v_j)$ are non-empty. Thus, both $u_i$ and $u_j$ are of type \rom{1}. More precisely,
    \[u_i = \alpha_{e_i}^{v_i} \text{ and } u_j = \alpha_{e_j}^{v_j},\]
    where $e_i$ and $e_j$ are again the corresponding edges in the support graphs $\sg{\Gamma}{v_i}$ and $\sg{\Gamma}{v_j}$ respectively.\par
    Also, $\Gamma - (\lk(v_i) \cap \lk(v_j))$ has $\{a_1, a_2\}$ as a component since both $b_1$ and $b_2$ are in $\lk(v_i) \cap \lk(v_j)$. As $v_i \not\sim v_j$, we get that $(v_i,v_j)$ is a \mbox{SIL-pair}. Furthermore, the edges in the support graphs of $v_i$ and $v_j$ are of the form $([v_j]_{v_i},\{a_1,a_2\})$ and $([v_i]_{v_j},\{a_1,a_2\})$. Hence, there is no edge between $u_i$ and $u_j$.
\end{description}
Thus, we indeed have that $\Lambda \cong \Theta$. We can conclude that
\[A_\Lambda \cong A_\Theta \cong \PSO(A_\Gamma),\]
which is what we wanted to show.
\end{proof}
We now prove the second lemma, which states that for this $\Gamma = \Gamma(\Lambda)$ we have that $\PSO(A_\Gamma)$ has finite index in $\Out(A_\Gamma)$.
\begin{proof}[Proof of Lemma \ref{lem:PSOAGammahasfiniteindex}]
In order to prove this, we want to use Theorem \ref{thm:condforfiniteindex}, so we need to show that $\lk(u) \subseteq \st(w)$ implies $u = w$.
\begin{table}[b]
    \centering
    \resizebox{\textwidth}{!}{
    \begin{tabular}{|c|c|c|c|c|c|c|c|c|c|}
        \hline
        & $\st(a_1)$ & $\st(a_2)$ & $\st(b_1)$ & $\st(b_2)$ & $\st(v_1)$ & $\st(v_l)$ & $\st(c_j)$ & $\st(d_1)$ & $\st(d_2)$\\\hline
        $\lk(a_1)$ & $\checkmark$ & $b_1$ & $a_2$ & $b_1$ & $a_2$ & $a_2$ & $a_2$ & $a_2$ & $a_2$ \\\hline
        $\lk(a_2)$ & $b_2$ & $\checkmark$ & $b_2$ & $a_1$ & $a_1$ & $a_1$ & $a_1$ & $a_1$ & $a_1$ \\\hline
        $\lk(b_1)$ & $d_1$ & $d_1$ & $\checkmark$ & $d_1$ & $a_1$ & $a_1$ & $a_1$ & $a_1$ & $a_1$ \\\hline
        $\lk(b_2)$ & $d_2$ & $d_2$ & $d_2$ & $\checkmark$ & $a_2$ & $a_2$ & $a_2$ & $a_2$ & $a_2$ \\\hline
        $\lk(v_1)$ & $b_2$ & $b_1$ & $b_2$ & $b_1$ & $\checkmark$ & $d_1$ & $b_1$ & $b_2$ & $b_1$ \\\hline
        $\lk(v_k)$ & $b_2$ & $b_1$ & $b_2$ & $b_1$ & $d_2$ & $*$ & $b_1$ & $b_2$ & $b_1$ \\\hline
        $\lk(c_i)$ & $d_1$ & $d_1$ & $d_2$ & $d_1$ & $d_2$ & $d_1$ & $*$ & $d_2$ & $d_1$ \\\hline
        $\lk(d_1)$ & $c_1$ & $c_1$ & $c_1$ & $c_1$ & $c_2$ & $*$ & $b_1$ & $\checkmark$ & $v_1$ \\\hline
        $\lk(d_2)$ & $c_1$ & $c_1$ & $c_1$ & $c_1$ & $c_2$ & $*$ & $b_2$ & $v_2$ & $\checkmark$ \\\hline
    \end{tabular}}
    \caption{Determining when $\text{lk(u)}$ is a subset of $\st(w)$}
    \label{tab:lkusubseteqstv}
\end{table}
We do this with Table \ref{tab:lkusubseteqstv}, which lists all pairs $(\lk(u), \st(w))$. If $\lk(u) \subseteq \st(w)$, we write a ``$\checkmark$'', otherwise we provide an element in $\lk(u) \setminus \st(w)$. Regarding the indices, $i$ and $j$ go from $1$ to $n$ and $k$ and $l$ go from $2$ to $n$. For the parts denoted with ``$*$'', we need to distinguish what the indices are, which we do in the following. All these special cases need the assumption $n \geq 3$. The problem with $n=2$ is that we only have two vertices $c_i$ and both of them are connected to all vertices of $\Lambda$. We also need this assumption in the table, namely that $c_2 \neq c_n$ and thus $c_2 \notin \st(v_1)$.
\begin{itemize}
    \item $\lk(v_k) \subseteq \st(v_l)$ is only true if $k = l$, otherwise there is a vertex of $\{c_1,\dots,c_n\}$ that is in $\lk(v_k) \setminus \st(v_l)$.
    \item $\lk(c_i) \subseteq \st(c_j)$ holds only if $i=j$, otherwise a vertex of $\{v_1,\dots,v_n\}$ is in $\lk(c_i) \setminus \st(c_j)$.
    \item $\lk(d_1) \not\subseteq \st(v_l)$ as $c_1$ or $c_n$ is not in $\st(v_l)$ since $l>1$.
    \item $\lk(d_2) \not\subseteq \st(v_l)$ by the same reasoning as for $\lk(d_1) \not\subseteq \st(v_l)$.
\end{itemize}
With the table and the special cases, we conclude that $\lk(u) \subseteq \st(w)$ implies $u = w$. As discussed in the beginning, this concludes the proof.
\end{proof}
Next, we provide a better construction to strengthen Theorem \ref{thm:anyraagisfinindsubgrofoutraag}.
\begin{thm}\label{thm:anyraagisfinindsubgrofoutraagnographauto}
For any graph $\Lambda$, there is a graph $\Gamma' = \Gamma'(\Lambda)$ such that
\begin{enumerate}
    \item $A_\Lambda \cong \PSO(A_{\Gamma'})$,
    \item $\PSO(A_{\Gamma'})$ has finite index in $\Out(A_{\Gamma'})$ and
    \item $\Gamma'$ has no non-trivial graph automorphisms.
\end{enumerate}
\end{thm}
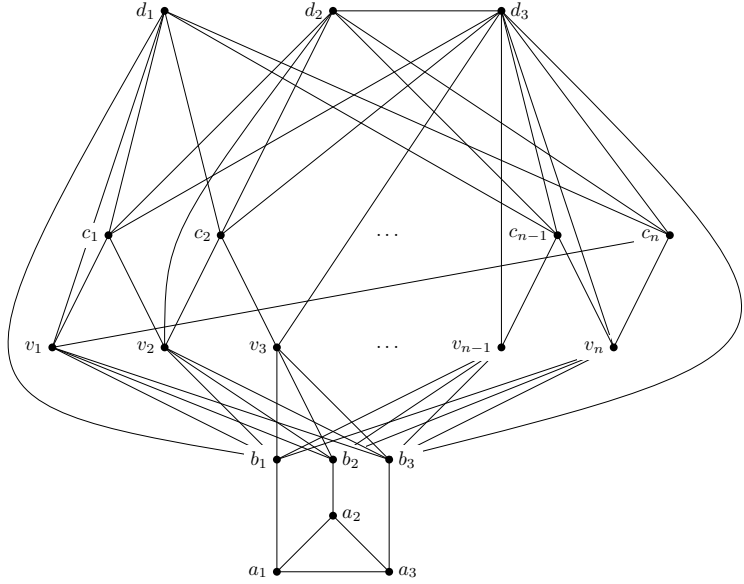
\begin{figure}
    \centering
    \resizebox{\textwidth}{!}{
    \centering
    \begin{tikzpicture}

    \coordinate (v1) at (0,0);
    \coordinate (v2) at (2,0);
    \coordinate (v3) at (4,0);
    \coordinate (Dots1) at (6,0);
    \coordinate (vn-1) at (8,0);
    \coordinate (vn) at (10,0);
    \coordinate (C1) at (1,2);
    \coordinate (C2) at (3,2);
    \coordinate (Dots2) at (6,2);
    \coordinate (Cn-1) at (9,2);
    \coordinate (Cn) at (11,2);
    \coordinate (D1) at (2,6);
    \coordinate (D3) at (8,6);
    \coordinate (D2) at (5,6);
    \coordinate (B1) at (4,-2);
    \coordinate (B3) at (6,-2);
    \coordinate (B2) at (5,-2);
    \coordinate (A1) at (4,-4);
    \coordinate (A3) at (6,-4);
    \coordinate (A2) at (5,-3);

    \node at (Dots1) {\dots};
    \node at (Dots2) {\dots};
    \node at (D1) [left = 0.4mm of D1, fill=white] {$d_1$};
    \node at (D2) [left = 0.4mm of D2, fill=white] {$d_2$};
    \node at (D3) [right = 0.4mm of D3, fill=white] {$d_3$};

    \draw (v1) -- (D1);
    \draw (v1) -- (Cn);
    \draw (v1) -- (B1);
    \draw (v2) -- (B1);
    \draw[] (B1) .. controls(-2,-1) .. (D1);
    \draw (vn-1) -- (B1);
    \draw (vn) -- (B1);
    \draw (vn) -- (B2);
    \draw[] (B3) .. controls(14,0) .. (D3);
    \draw (vn) -- (D3);
    \draw (vn-1) -- (B2);
    \draw (vn-1) -- (D3);
    \draw (vn-1) -- (B3);
    \draw (vn) -- (B3);

    \node at (v1) [left = 0.4mm of v1, fill=white] {$v_1$};
    \node at (v2) [left = 0.4mm of v2, fill=white] {$v_2$};
    \node at (v3) [left = 0.4mm of v3, fill=white] {$v_3$};
    \node at (vn-1) [left = 0.4mm of vn-1, fill=white] {$v_{n-1}$};
    \node at (vn) [left = 0.4mm of vn, fill=white] {$v_n$};
    \node at (C1) [left = 0.4mm of C1, fill=white] {$c_1$};
    \node at (C2) [left = 0.4mm of C2, fill=white] {$c_2$};
    \node at (Cn-1) [left = 0.4mm of Cn-1, fill=white] {$c_{n-1}$};
    \node at (Cn) [left = 0.4mm of Cn, fill=white] {$c_n$};
    \node at (B1) [left = 0.4mm of B1, fill=white] {$b_1$};
    \node at (B2) [right = 0.4mm of B2, fill = white] {$b_2$};
    \node at (B3) [right = 0.4mm of B3, fill = white] {$b_3$};
    \node at (A1) [left = 0.4mm of A1, fill=white] {$a_1$};
    \node at (A2) [right = 0.4mm of A2, fill = white] {$a_2$};
    \node at (A2) [right = 0.4mm of A3, fill = white] {$a_3$};

    \draw (v1) -- (C1);
    \draw (v2) -- (C1);
    \draw (v2) -- (C2);
    \draw (v3) -- (C2);
    \draw (vn-1) -- (Cn-1);
    \draw (vn) -- (Cn-1);
    \draw (vn) -- (Cn);
    \draw[] (v2) .. controls(2,2) .. (D2);
    \draw (v3) -- (D3);
    \draw (C1) -- (D1);
    \draw (C2) -- (D1);
    \draw (Cn-1) -- (D1);
    \draw (Cn) -- (D1);
    \draw (C1) -- (D2);
    \draw (C2) -- (D2);
    \draw (Cn-1) -- (D2);
    \draw (Cn) -- (D2);
    \draw (C1) -- (D3);
    \draw (C2) -- (D3);
    \draw (Cn-1) -- (D3);
    \draw (Cn) -- (D3);
    \draw (v3) -- (B1);
    \draw (v1) -- (B2);
    \draw (v2) -- (B2);
    \draw (v3) -- (B2);
    \draw (v1) -- (B3);
    \draw (v2) -- (B3);
    \draw (v3) -- (B3);
    \draw (B1) -- (A1);
    \draw (B2) -- (A2);
    \draw (B3) -- (A3);
    \draw (A1) -- (A2);
    \draw (A1) -- (A3);
    \draw (A2) -- (A3);
    \draw (D2) -- (D3);

    \fill (v1) circle (2pt);
    \fill (v2) circle (2pt);
    \fill (v3) circle (2pt);
    \fill (vn-1) circle (2pt);
    \fill (vn) circle (2pt);
    \fill (C1) circle (2pt);
    \fill (C2) circle (2pt);
    \fill (Cn-1) circle (2pt);
    \fill (Cn) circle (2pt);
    \fill (D1) circle (2pt);
    \fill (D2) circle (2pt);
    \fill (D3) circle (2pt);
    \fill (B1) circle (2pt);
    \fill (B2) circle (2pt);
    \fill (B3) circle (2pt);
    \fill (A1) circle (2pt);
    \fill (A2) circle (2pt);
    \fill (A3) circle (2pt);
    \end{tikzpicture}
    }
    \caption{Construction of $\Gamma'$}
    \label{fig:construction_new}
\end{figure}
We again assume $n \geq 3$; the other cases are covered in Appendix \ref{app:smallgraphs}. We get $\Gamma' = \Gamma'(\Lambda)$ from $\Lambda$ by adding the vertices $a_1$, $a_2$, $a_3$, $b_1$, $b_2$, $b_3$, $c_1$, \dots, $c_n$, $d_1$, $d_2$ and $d_3$ and the edges that are shown in Figure \ref{fig:construction_new}: The vertices $d_i$ are connected to all vertices $c_j$. Furthermore, $d_1$ is connected to $v_1$ and $b_1$, $d_2$ to $v_2$, $d_3$ to $v_k$ for $k > 2$ and to $b_3$ and we have the edge $\{d_2, d_3\}$. Note that $d_2$ is not connected to $b_2$. As for the construction of $\Gamma$ for Theorem \ref{thm:anyraagisfinindsubgrofoutraag}, $c_j$ is connected to $v_j$ and $v_{j+1}$ ($v_{n+1} \coloneqq v_1$), the vertices $v_k$ are connected to $b_1$, $b_2$ and $b_3$ and we have all edges than are present in $\Lambda$ also in $\Gamma'$. Finally, we have the edges $\{b_1, a_1\}$, $\{b_2, a_2\}$, $\{b_3, a_3\}$, $\{a_1, a_2\}$, $\{a_1, a_3\}$ and $\{a_2, a_3\}$. The proof of the theorem now follows from the following three lemmas.
\begin{lem}\label{lem:PSOAGamma'isALambda}
For this graph $\Gamma' = \Gamma'(\Lambda)$, we have that $\PSO(A_{\Gamma'})$ is isomorphic to $A_\Lambda$.
\end{lem}
\begin{lem}\label{lem:PSOAGamma'hasfiniteindex}
For this graph $\Gamma' = \Gamma'(\Lambda)$, $\PSO(A_{\Gamma'})$ has finite index in $\Out(A_{\Gamma'})$.
\end{lem}
\begin{lem}\label{lem:Gamma'hasnographautos}
For this graph $\Gamma' = \Gamma'(\Lambda)$, $\Gamma'$ has no non-trival graph automorphisms.
\end{lem}
Lemmas \ref{lem:PSOAGamma'isALambda} and \ref{lem:PSOAGamma'hasfiniteindex} can be shown as Lemmas \ref{lem:PSOAGammaisALambda} and \ref{lem:PSOAGammahasfiniteindex} using Theorems \ref{thm:condforpsobeingaraag} and \ref{thm:condforfiniteindex}.
Thus, we only comment on the proof of Lemma \ref{lem:Gamma'hasnographautos}.
\begin{proof}[Proof sketch of Lemma \ref{lem:Gamma'hasnographautos}]
To show this, one uses the fact that a graph automorphism sends adjacent vertices to adjacent vertices and in particular preserves the degree of every vertex. One can first show that the levels of $\Gamma'$ are fixed, i.e. the vertices $a_i$ are mapped to vertices $a_j$ by any automorphism and so on. This is done from bottom to top, i.e. starting with the $a_i$, then the $b_j$ and so on. Next, one can show that this implies that the vertices also need be fixed pointwise by any automorphism, which concludes the proof.
\end{proof}
We finish this section with a corollary about the structure of the quotient $\Out(A_{\Gamma'})/\PSO(A_{\Gamma'})$ and compute the index of $\PSO(A_{\Gamma'})$ in $\Out(A_{\Gamma'})$. We still assume $n \geq 3$. One can prove a similar version of this corollary for $n < 3$ using the constructions in Appendix \ref{app:smallgraphs}.
\begin{cor}\label{cor:structureofquotient}
For the graph $\Gamma' = \Gamma'(\Lambda)$ as described before,
\[\Out(A_{\Gamma'})/\PSO(A_{\Gamma'}) \cong \left(\Z/2\Z\right)^N\]
and thus the index of $\PSO(A_{\Gamma'})$ in $\Out(A_{\Gamma'})$ is $2^N$, where $N = 2n+9$ is the number of vertices of $\Gamma'$.
\end{cor}
\begin{proof}
Using the generators of $\Aut(A_{\Gamma'})$ from Section \ref{sec:generators} one can show that if $\PSO(A_{\Gamma'})$ has finite index in $\Out(A_{\Gamma'})$ one has
\[\Out(A_{\Gamma'})/\PSO(A_{\Gamma'})\cong \Inv \rtimes \Per.\]
Here, $\Inv$ is the subgroup of $\Out(A_{\Gamma'})$ generated by inversions and $\Per$ is the subgroup generated by $\Gamma'$-legal permutations. Since $\Gamma'$ has no non-trivial graph automorphisms, the group $\Per$ of $\Gamma'$-legal permutations contains only the identity. Using that $\Inv \cong (\Z/2\Z)^N$ for $N$ the number of vertices of $\Gamma'$, we get that
\[\Out(A_{\Gamma'})/\PSO(A_{\Gamma'}) \cong \Inv \cong (\Z/2\Z)^N.\]
As $|\left(\Z/2\Z\right)^N| = 2^N$, we also get that the index of $\PSO(A_{\Gamma'})$ in $\Out(A_{\Gamma'})$ is $2^N$. To complete the proof, it remains to argue why $N$ satisfies the claimed equality. The graph $\Gamma'$ as constructed above has $2n+9$ vertices since $\Lambda$ has $n$ vertices and we add $n+9$ vertices.
\end{proof}

\section{Conclusion and outlook}\label{chap:conclusion}
We conclude this paper by summarising what we achieved and pointing out some possible directions for further research. In the paper \cite{DayWade}, Day and Wade gave a condition for when $\PSO(A_\Gamma)$ is a right-angled Artin group and which one it is. In this paper, we asked the question which right-angled Artin groups occur as (finite-index) subgroups of $\Out(A_\Gamma)$ when we vary $\Gamma$. Using the condition by Day--Wade and another condition by Wade--Brück, we answered this question in Theorem \ref{thm:anyraagisfinindsubgrofoutraag} by showing that every right-angled Artin group occurs as $\PSO(A_\Gamma)$ for some graph $\Gamma$ for which $\PSO(A_\Gamma)$ has finite index in $\Out(A_\Gamma)$.\par
A possible direction for further research is to simplify the construction or to impose further conditions as we did in Theorem \ref{thm:anyraagisfinindsubgrofoutraagnographauto}, where we additionally wanted that the graph $\Gamma$ has no non-trivial graph automorphisms. The advantage of this compared to the theorem above is that it simplifies the structure of the quotient $\Out(A_\Gamma)/\PSO(A_\Gamma)$. For other simplifications, one could try to reduce the number of vertices or to reduce the index of $\PSO(A_\Gamma)$ in $\Out(A_\Gamma)$.\par
Another interesting topic to think about is to go beyond the subgroup of pure symmetric outer automorphisms and investigate for which graphs $\Gamma$ a fixed right-angled Artin group $A_\Lambda$ is a (finite-index) subgroup of the outer automorphism group of $A_\Gamma$. As mentioned above, we showed in this paper, using the subgroup of pure symmetric outer automorphisms, that every right-angled Artin group $A_\Lambda$ occurs as a finite-index subgroup of $\Out(A_\Gamma)$ for some graph $\Gamma$. Also, for fixed $\Lambda$, one can, using the condition of Day--Wade, answer the question when $\PSO(A_\Gamma)$ is isomorphic to $A_\Lambda$. Namely, $A_\Lambda \cong \PSO(A_\Gamma)$ if and only if the graph $\Gamma$ satisfies that all its support graphs are forests and the graph described in Theorem \ref{thm:condforpsobeingaraag} is isomorphic to $\Lambda$. It would be interesting to generalise this to other examples that do not depend on the subgroup of pure symmetric outer automorphisms and maybe for fixed $\Lambda$ give conditions that a graph $\Gamma$ needs to satisfy so that $A_\Lambda$ is a finite-index subgroup of $\Out(A_\Gamma)$. Such results, even just for some graphs $\Lambda$, might give insights about how right-angled Artin groups interact with their outer automorphism groups and help to find more general results. However, a difficulty with this approach is that one cannot look at a fixed subgroup of $\Out(A_\Gamma)$ but needs to consider all subgroups of $\Out(A_\Gamma)$ at the same time.\par
Another interesting path one could take is to change the point of view and return to the question that Day and Wade asked; see \cite[Question 1.1]{DayWade}. This is to not ask for fixed $\Lambda$ when $A_\Lambda$ is a (finite-index) subgroup of the outer automorphism group of another right-angled Artin group, but to ask for fixed $\Gamma$ which right-angled Artin groups occur as a (finite-index) subgroup of $\Out(A_\Gamma)$ and if there are any at all. When restricted to the subgroup of pure symmetric outer automorphisms, this is again answered by the condition of Day--Wade, but it would be interesting to generalise this to arbitrary (finite-index) subgroups.\par
Finally, one could also try to find examples when the automorphism group $\Aut(A_\Gamma)$ has right-angled Artin groups as (finite-index) subgroups. One could ask the same questions as the ones asked by Day--Wade and in this paper. There is already a result known in this direction. Namely, Charney--Ruane--Stambaugh--Vijayan prove in \cite[Theorem 3.6]{CharneyetalPSA} that the group $\PSA(A_\Gamma)$ is a right-angled Artin group if the graph $\Gamma$ has no separating intersection of links. For $\PSA(A_\Gamma)$, one can still use the condition for finite index as in Theorem \ref{thm:condforfiniteindex} since $\PSO(A_\Gamma)$ has finite index in $\Out(A_\Gamma)$ if and only if $\PSA(A_\Gamma)$ has finite index in $\Aut(A_\Gamma)$. To show this, one needs that $\Inn(A_\Gamma)$ is a normal subgroup of $\Aut(A_\Gamma)$ and that $\Inn(A_\Gamma) \subseteq \PSA(A_\Gamma)$. Then, the natural map
\[(\Aut(A_\Gamma)/\Inn(A_\Gamma))/(\PSA(A_\Gamma)/\Inn(A_\Gamma)) \longrightarrow \Aut(A_\Gamma)/\PSA(A_\Gamma)\]
is well-defined and bijective. Thus, since
\[\Out(A_\Gamma)/\PSO(A_\Gamma) = (\Aut(A_\Gamma)/\Inn(A_\Gamma))/(\PSA(A_\Gamma)/\Inn(A_\Gamma)),\]
the index of $\PSO(A_\Gamma)$ in $\Out(A_\Gamma)$ is the same as the index of $\PSA(A_\Gamma)$ in $\Aut(A_\Gamma)$. One could try to use \cite{CharneyetalPSA} and the finite-index condition from Theorem \ref{thm:condforfiniteindex} to answer the question which right-angled Artin groups occur as a (finite-index) subgroup of $\Aut(A_\Gamma)$ for some $\Gamma$. Similarly to this paper, one could use computer programs to find examples and then try to generalise these.

\begin{appendices}
\section{Small graphs}\label{app:smallgraphs}
In this appendix, we give examples for the graphs with less than three vertices to complete the proofs of Theorems \ref{thm:anyraagisfinindsubgrofoutraag} and \ref{thm:anyraagisfinindsubgrofoutraagnographauto}. For every such graph $\Lambda$, we give a graph $\Gamma$ such that $\PSO(A_\Gamma) \cong A_\Lambda$, $\PSO(A_\Gamma)$ has finite index in $\Out(A_\Gamma)$ and $\Gamma$ has no non-trivial graph automorphisms. These graphs were found using computer programs that can be found on \cite{Code}. This code can also be used to check that they satisfy the conditions of Theorems \ref{thm:anyraagisfinindsubgrofoutraag} and \ref{thm:anyraagisfinindsubgrofoutraagnographauto}. In this appendix, we only sketch why the graph $\Theta$ as defined in Theorem \ref{thm:condforpsobeingaraag} is indeed~$\Lambda$.\par
Define $\Lambda_1$ as the graph with one vertex, $\Lambda_2$ as the graph with two vertices and no edge and $\Lambda_3$ as the graph with two vertices and an edge. Then one can show that the graphs $\Gamma_1$, $\Gamma_2$ and $\Gamma_3$ as in Figures \ref{fig:smallgraphsonevertex}, \ref{fig:smallgraphstwoverticesnoedge} and \ref{fig:smallgraphstwoverticesoneedge} satisfy the conditions stated above, i.e. $\PSO(A_{\Gamma_i}) \cong A_{\Lambda_i}$ (using Theorem \ref{thm:condforpsobeingaraag}), $\PSO(A_{\Gamma_i})$ has finite index in $\Out(A_{\Gamma_i})$ (using Theorem \ref{thm:condforfiniteindex}) and $\Gamma_i$ has no non-trivial graph automorphisms.\par
As mentioned above, we only sketch $\PSO(A_{\Gamma_i}) \cong A_{\Lambda_i}$; for the complete arguments, the computer programs from \cite{Code} can be used. For $\Gamma_1$, $v_9$ is the only vertex with a separating star. The support graph consists of two vertices without an edge and thus we have only one vertex $\beta_1^{v_9}$ in the graph $\Theta_1$ as defined in Theorem \ref{thm:condforpsobeingaraag}. For all other vertices $v_i$, $\Gamma_1 - \st{v_i}$ is connected and hence they do not define vertices in $\Theta_1$. Thus, we have $\Theta_1 \cong \Lambda_1$. For $\Gamma_2$, the vertices $v_6$ and $v_7$ are the only ones with separating stars. Both their support graphs have two vertices and an edge, so the graph $\Theta_2$ has two vertices $\alpha_{e_6}^{v_6}$ and $\alpha_{e_7}^{v_7}$, where $e_6$ and $e_7$ are the edges in the support graphs of $v_6$ and $v_7$ respectively. These two vertices satisfy the condition for when there is no edge between vertices of type I and thus the graph $\Theta_2$ has no edge. Thus, we have $\Theta_2 \cong \Lambda_2$. For $\Gamma_3$, the vertices $v_8$ and $v_9$ are the only vertices with separating stars and their support graphs have both two vertices and no edge. Thus, in $\Theta_3$, we have two vertices $\beta_1^{v_8}$ and $\beta_1^{v_9}$ and an edge. Hence, we have $\Theta_3 \cong \Lambda_3$. Using Theorem \ref{thm:condforpsobeingaraag}, we get that $\PSO(A_{\Gamma_i}) \cong A_{\Lambda_i}$.
\begin{figure}[h]
    \centering
    \begin{subfigure}[b]{0.49\textwidth}
        \centering
        \begin{tikzpicture}[scale=1.25]
            \coordinate (1) at (-1,2);
            \coordinate (2) at (0,2);
            \coordinate (3) at (1,2);
            \coordinate (4) at (-2,1);
            \coordinate (5) at (-1,1);
            \coordinate (6) at (0,1);
            \coordinate (7) at (1,1);
            \coordinate (8) at (2,1);
            \coordinate (9) at (0,0);
            \coordinate (10) at (-1,-1);
            \coordinate (11) at (1,-1);
            
            \draw (1) -- (4);

            \node at (4) [above=0.05, fill=white] {$v_4$};

            \draw (1) -- (5);
            \draw (2) -- (6);
            \draw (3) -- (4);
            \draw (3) -- (6);
            \draw (4) -- (11);            
            \draw (8) -- (10);

            \node at (1) [above=0.05, fill=white] {$v_1$};
            \node at (2) [above=0.05, fill=white] {$v_2$};
            \node at (3) [above=0.05, fill=white] {$v_3$};
            \node at (5) [above=0.05, fill=white] {$v_5$};
            \node at (6) [above=0.05, fill=white] {$v_6$};
            \node at (7) [above=0.05, fill=white] {$v_7$};
            \node at (8) [above=0.05, fill=white]{$v_8$};
            \node at (9) [below=0.05, fill=white] {$v_9$};
            \node at (10) [below=0.05, fill=white] {$v_{10}$};
            \node at (11) [below=0.05, fill=white] {$v_{11}$};
            
            \draw (1) -- (2);
            \draw (2) -- (3);
            \draw (2) -- (8);
            \draw (3) -- (8);
            \draw (4) -- (9);
            \draw (4) -- (10);
            \draw (5) -- (6);
            \draw (5) -- (9);
            \draw (5) -- (10);
            \draw (6) -- (7);
            \draw (6) -- (9);
            \draw (7) -- (8);
            \draw (7) -- (9);
            \draw (7) -- (11);
            \draw (8) -- (9);
            \draw (10) -- (11);
            
            \fill (1) circle (1.5pt);
            \fill (2) circle (1.5pt);
            \fill (3) circle (1.5pt);
            \fill (4) circle (1.5pt);
            \fill (5) circle (1.5pt);
            \fill (6) circle (1.5pt);
            \fill (7) circle (1.5pt);
            \fill (8) circle (1.5pt);
            \fill (9) circle (1.5pt);
            \fill (10) circle (1.5pt);
            \fill (11) circle (1.5pt);
        \end{tikzpicture}
        \caption{The graph $\Gamma_1$}
        \label{fig:smallgraphsonevertex}
    \end{subfigure}
    \hfill
    \begin{subfigure}[b]{0.49\textwidth}
        \centering
        \begin{tikzpicture}[scale=1.25]
            \coordinate (1) at (-1,2);
            \coordinate (2) at (1,2);
            \coordinate (3) at (-2,1);
            \coordinate (4) at (0,1);
            \coordinate (5) at (2,1);
            \coordinate (6) at (-1,0);
            \coordinate (8) at (-1,-1);
            \coordinate (7) at (1,0);
            \coordinate (9) at (1,-1);
            
            \draw (1) -- (3);

            \node at (1) [above=0.05, fill=white] {$v_1$};
            \node at (2) [above=0.05, fill=white] {$v_2$};
            \node at (3) [above=0.05, fill=white] {$v_3$};
            \node at (4) [above=0.05, fill=white] {$v_4$};
            \node at (5) [above=0.05, fill=white] {$v_5$};
            \node at (7) [below=0.05, fill=white] {$v_7$};
            \node at (6) [below=0.05, fill=white] {$v_6$};
            \node at (8) [below=0.05, fill=white] {$v_8$};
            \node at (9) [below=0.05, fill=white] {$v_9$};
    
            \draw (1) -- (2);
            \draw (1) -- (5);
            \draw (1) -- (6);
            \draw (2) -- (3);
            \draw (2) -- (4);
            \draw (2) -- (7);
            \draw (3) -- (6);
            \draw (3) -- (7);
            \draw (3) -- (8);
            \draw (4) -- (6);
            \draw (4) -- (7);
            \draw (4) -- (8);
            \draw (4) -- (9);
            \draw (5) -- (6);
            \draw (5) -- (7);
            \draw (5) -- (9);
            \draw (8) -- (9);
            
            \fill (1) circle (1.5pt);
            \fill (2) circle (1.5pt);
            \fill (3) circle (1.5pt);
            \fill (4) circle (1.5pt);
            \fill (5) circle (1.5pt);
            \fill (6) circle (1.5pt);
            \fill (7) circle (1.5pt);
            \fill (8) circle (1.5pt);
            \fill (9) circle (1.5pt);
            \end{tikzpicture}
        \caption{The graph $\Gamma_2$}
        \label{fig:smallgraphstwoverticesnoedge}
    \end{subfigure}\\
    \vspace{0.4cm}
    \begin{subfigure}[b]{0.49\textwidth}
        \centering
        \begin{tikzpicture}[scale=1.25]
            \coordinate (1) at (-2,3);
            \coordinate (2) at (2,3);
            \coordinate (3) at (-2,2);
            \coordinate (4) at (-1,2);
            \coordinate (5) at (0,2);
            \coordinate (6) at (1,2);
            \coordinate (7) at (2,2);
            \coordinate (8) at (-1,1);
            \coordinate (9) at (1,1);
            \coordinate (10) at (-2,0);
            \coordinate (11) at (2,0);

            \draw (1) -- (3);
            \draw (2) -- (7);
            
            \node at (4) [above=0.05, fill=white] {$v_4$};
            \node at (6) [above=0.05, fill=white] {$v_6$};

            \draw (2) -- (4);
            \draw (1) -- (6);

            \node at (1) [above=0.05, fill=white] {$v_1$};
            \node at (2) [above=0.05, fill=white] {$v_2$};
            \node at (3) [above=0.05, fill=white] {$v_3$};
            \node at (5) [above=0.05, fill=white] {$v_5$};
            \node at (7) [above=0.05, fill=white] {$v_7$};
            \node at (8) [below=0.05, fill=white] {$v_8$};
            \node at (9) [below=0.05, fill=white] {$v_9$};
            \node at (10) [below=0.05, fill=white] {$v_{10}$};  
            \node at (11) [below=0.05, fill=white] {$v_{11}$};
    
            \draw (1) -- (2);
            \draw (1) -- (5);
            \draw (2) -- (5);
            \draw (3) -- (4);
            \draw (3) -- (8);
            \draw (4) -- (5);
            \draw (4) -- (8);
            \draw (4) -- (9);
            \draw (4) -- (10);
            \draw (5) -- (6);
            \draw (5) -- (9);
            \draw (6) -- (8);
            \draw (6) -- (9);
            \draw (6) -- (11);
            \draw (7) -- (8);
            \draw (7) -- (9);
            \draw (7) -- (11);
            \draw (8) -- (9);
            \draw (10) -- (11);
    
            \fill (1) circle (1.5pt);
            \fill (2) circle (1.5pt);
            \fill (3) circle (1.5pt);
            \fill (4) circle (1.5pt);
            \fill (5) circle (1.5pt);
            \fill (6) circle (1.5pt);
            \fill (7) circle (1.5pt);
            \fill (8) circle (1.5pt);
            \fill (9) circle (1.5pt);
            \fill (10) circle (1.5pt);  
            \fill (11) circle (1.5pt);
        \end{tikzpicture}
        \caption{The graph $\Gamma_3$}
        \label{fig:smallgraphstwoverticesoneedge}
        \end{subfigure}
        \caption{Proving Theorems \ref{thm:anyraagisfinindsubgrofoutraag} and \ref{thm:anyraagisfinindsubgrofoutraagnographauto} for the small graphs}
        \label{fig:smallgraphs}
\end{figure}
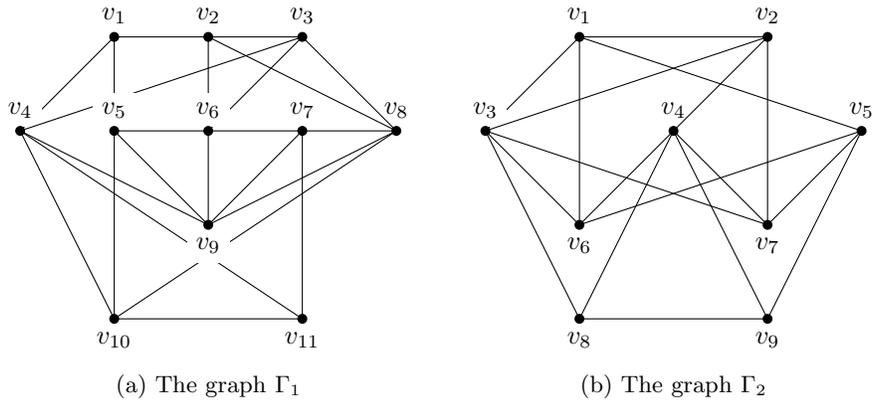
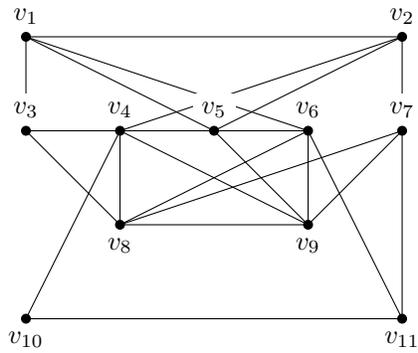
\end{appendices}

\clearpage
\printbibliography

\bigskip
Affiliation:\\
Department of Mathematics\\
ETH Zurich\\
8092 Zurich, Switzerland

\bigskip
Current address:\\
Manuel Wiedmer\\
Institute for Theoretical Computer Science\\
Department of Computer Science\\
ETH Zurich\\
8092 Zurich, Switzerland\\
\href{mailto:manuel.wiedmer@inf.ethz.ch}{manuel.wiedmer@inf.ethz.ch}

\end{document}